\documentclass[11pt, a4paper]{amsart}

\usepackage{amsmath,amsthm}
\usepackage{amssymb, enumitem, mathrsfs, verbatim, bm, esint}
\usepackage{amstext}
\usepackage[top=3cm,bottom=2.5cm,left=2.5cm,right=2.5cm]{geometry}

\usepackage{hyperref}
\usepackage{xcolor}

\newtheorem{theorem}{\bf Theorem}[section]

\newtheorem{lemma}{\bf Lemma}[section]
\newtheorem{definition}{\bf Definition}[section]

\newcommand{\N}{\mathbb N}

\numberwithin{equation}{section}
\newcommand{\diver}{\operatorname{div}}

\DeclareMathOperator*{\supp}{supp}
\newcommand{\dert}{\partial_t}

\newcommand{\thetah}{\hat\thet}

\newcommand{\vc}[1]{{\bf #1}}
\newcommand{\vv}{\vc{v}}

\newcommand{\vw}{\vc{w}}
\newcommand{\thet}{\vartheta}

\renewcommand{\S}{\mathbf{S}}
\newcommand{\D}{\mathbf{D}}

\newcommand{\dt}{\,{\rm d} t }
\newcommand{\dx}{\,{\rm d} {x}}
\newcommand{\dxdt}{\dx  \dt}

\newcommand{\G}{\mathcal{G}}
\newcommand{\al}{\alpha}

\newcommand{\intO}[1]{\int_{\Omega} #1 \ \dx}
\newcommand{\intTO}[1]{\int_0^T \int_{\Omega} #1 \ \dxdt}

\newcommand{\intTOI}[1]{\int_0^{\infty} \int_{\Omega} #1 \ \dxdt}

\newcommand{\tmin}{\underline{\thet}}
\newcommand{\tmax}{\overline{\thet}}

\begin{document}

\title[On the exponential decay in time of solutions to a Navier-Stokes-Fourier system]{On the exponential decay in time of solutions to a~generalized Navier-Stokes-Fourier system}
\author[A.~Abbatiello]{Anna Abbatiello}
\address{Anna Abbatiello.
Sapienza University of Rome, Department of Mathematics ``G. Castelnuovo", Piazzale Aldo Moro 5, 00185 Rome, Italy.}   \email{\tt anna.abbatiello@uniroma1.it}

\author[M.~Bul\'{i}\v{c}ek]{Miroslav Bul\'i\v{c}ek}
\address{Miroslav Bul\'i\v{c}ek. Mathematical Institute, Faculty of Mathematics and Physics, Charles University, Sokolovsk\'{a} 83, 18675 Prague, Czech Republic.}
\email{\tt mbul8060@karlin.mff.cuni.cz}

\author[P.~Kaplick\'{y}]{Petr Kaplick\'{y}}
\address{Petr Kaplick\'{y}. Charles University, Faculty of Mathematics and Physics, Department of Mathematical Analysis, Sokolovsk\'{a} 83, 18675 Prague, Czech Republic.}
 \email{\tt kaplicky@karlin.mff.cuni.cz}


\keywords{Navier--Stokes--Fourier equations, stability, equilibrium, non-Newtonian fluids}
\subjclass[2020]{35Q30, 35K61, 76E30, 37L15}
\thanks{A.~Abbatiello has been supported by the ERC-STG Grant n. 759229 HiCoS. Also, the research activity of A.~Abbatiello is performed under the auspices of GNAMPA - INdAM. M. Bul\'{\i}\v{c}ek and P. Kaplick\'{y} acknowledge the support of the project  No. 20-11027X financed by Czech Science Foundation (GA\v{C}R)}

\begin{abstract}
We consider a non-Newtonian incompressible heat conducting fluid with prescribed nonuniform temperature on the boundary and with the no-slip boundary conditions for the velocity. We assume no external body forces. For the power-law like models with the power law index bigger than $11/5$ in three dimensions, we identify a class of solutions fulfilling the entropy equality and converging to the equilibria exponentially in a proper metric. In fact, we show the existence of a Lyapunov functional for the problem. Consequently, the steady solution is nonlinearly stable and attracts all suitable weak solutions.
\end{abstract}

\maketitle


\section{Formulation of the problem and the main result}
We consider flows of an incompressible heat conducting non-Newtonian fluid in a three dimensional Lipschitz domain $\Omega\subset \mathbb{R}^3$ subjected to the homogeneous Dirichlet boundary condition for the velocity field and to the spatially inhomogeneous Dirichlet boundary condition for the temperature, i.e., we do not consider the transfer of mass through the boundary but we allow the transfer of the heat/energy. Our principal goal is to show that for a reasonable class of models, the flow converges to the unique stationary solution provided that the external body forces are equal to zero. Moreover, we want to have a proper rate of the convergence and also if possible to show a kind of stability of the equilibria.

To be more precise, we consider a couple $(\vv, \thet)$ that solves the following system of equations
\begin{align}
\dert \vv + \diver(\vv \otimes \vv) -\diver \S(\thet,\D \vv) &= - \nabla p,\label{MB1}\\
\diver \vv &= 0,\label{MB2}\\
\dert \thet + \diver (\vv \thet) - \diver (\kappa(\thet)\nabla \thet) &= \S(\thet,\D\vv) : \D \vv \label{MB3}
\end{align}
in $Q:=(0,\infty)\times \Omega$. The system \eqref{MB1}--\eqref{MB3} is completed by the initial conditions
\begin{equation}\label{MB4}
\vv(0, \cdot)=\vv_0, \quad \thet(0, \cdot)=\thet_0 \qquad \textrm{ in } \Omega
\end{equation}
and by the boundary conditions
\begin{equation}
\label{MB5}
\vv=\vc{0}, \quad \thet=\thetah   \qquad \textrm{ on } (0, \infty)\times \partial \Omega.
\end{equation}
Here, $\vv$ is the velocity field, $p$ is the pressure and $\thet$ is the temperature. The material is described by the constitutively determined part of the Cauchy stress $\S$, that is supposed to be a tensorial function of the temperature $\thet$ and the symmetric velocity gradient $\D\vv:=(\nabla \vv +(\nabla \vv)^T)/2$, i.e., $\S:\mathbb{R}_+ \times \mathbb{R}^{3\times 3}_{\textrm{sym}} \to \mathbb{R}^{3\times 3}_{\textrm{sym}}$, by the heat conductivity $\kappa:\mathbb{R}_+ \to \mathbb{R}_+$, and by the heat capacity. For simplicity and for the clarity of the presentation we consider the heat capacity being equal to one.

Further, we assume that $\kappa$ and $\S$ are continuous mappings and that there exist constants $0<\underline{\kappa}<\overline{\kappa}<\infty$, parameter $p\in (1,\infty)$ and $\delta\in (0,1]$ such that for all $(\thet,\D)\in\mathbb{R}_+\times\mathbb{R}^{3\times 3}_{\textrm{sym}}$
\begin{equation}
\label{MBBounds}
\begin{gathered}
\underline{\kappa}\le \kappa(\thet)\le \overline{\kappa}, \qquad |\S(\thet,\D)|\le \overline{\kappa}(\delta + |\D|)^{p-2},\\
\underline{\kappa}(\delta + |\D|)^{p-2}|\D|^2 \le \S(\thet,\D):\D.
\end{gathered}
\end{equation}

For the \emph{stationary} boundary data $\thetah$ we consider, that it is extended inside $\Omega$ such that $0<\tmin\le \thetah \le \tmax$ in $\Omega$ for some $\tmin,\tmax\in\mathbb{R}_+$ and it solves
\begin{equation}\label{MB6}
-\diver(\kappa(\thetah)\nabla \thetah)=0\quad\mbox{in $\Omega$.}
\end{equation}
The initial conditions are considered to have the finite energy, i.e.,
\begin{equation}\label{MB7}
\intO{\frac12|\vv_0|^2 + \thet_0} <\infty
\end{equation}
and to satisfy the standard compatibility conditions
\begin{equation}\label{MB7.5}
\diver \vv_0 =0, \; \thet_0 >\tmin \textrm{ in } \Omega \qquad \textrm{ and } \qquad \vv_0  \cdot \vc{n} = 0 \textrm{ on } \partial\Omega,
\end{equation}
where $\vc{n}$ denotes the unit outward normal vector on $\partial \Omega$. Note here, that \eqref{MB1} is the incompressibility constraint, \eqref{MB2} is the balance of linear momentum and \eqref{MB3} is the balance of internal energy. The balance of internal energy is sometimes completed or replaced by the entropy (in)equality
\begin{equation}\label{MB8}
\dert \eta + \diver (\vv \eta) - \diver (\kappa(\thet)\nabla \eta ) = \frac{\S(\thet,\D\vv) : \D \vv}{\thet} + \kappa(\thet)|\nabla \eta|^2,
\end{equation}
where $\eta:=\ln \thet$ denotes the entropy. Note that \eqref{MB3} and \eqref{MB8} are formally equivalent. However, such equivalence holds true only if the solution satisfies certain regularity properties. In this paper, we select certain uniform integrability criteria, see \eqref{limit-k}, that will guarantee such equivalence.
Solutions satisfying this condition will be called \emph{proper} solutions.
The existence of solutions satisfying \eqref{MB1}--\eqref{MB3} and \eqref{MB8} for any initial data fulfilling \eqref{MB7}--\eqref{MB7.5} is established in~\cite{ABK} provided that $\S$ is monotone with respect to $\D$ and\footnote{The existence of a weak solution for $p<11/5$ is studied e.g. in \cite{BuMaRa09,MaZa18} with one proviso, the entropy equality~\eqref{MB8} and the internal energy equality~\eqref{MB3} are replaced by inequalities and the system is completed by the balance of the global energy. For $p\ge 11/5$, see also e.g. \cite{Cons}.} $p\ge 11/5$ in three dimensions. The main result of the paper can be summarized as follows:

\bigskip

\newtheorem*{theorem*}{Main theorem}

\begin{theorem*}Let $\S$ and $\kappa$ satisfy \eqref{MBBounds} with $p\ge 11/5$. Then there exists $\lambda>0$ and a~function $(\vv,\thet,\thetah)\mapsto L(\vv,\thet,\thetah)$ that is nonnegative, strictly convex with respect to $\vv$ and $\thet$ and vanishing only if $\vv=0$ and $\thet=\thetah$ such that for all $\tau>\sigma\ge 0$ and for all proper solutions to \eqref{MB1}--\eqref{MB3} and \eqref{MB8} there holds
\begin{equation}\label{MB-res}
\intO{L(\vv(\tau),\thet(\tau),\thetah)}\le e^{\lambda (\sigma-\tau)}\intO{L(\vv(\sigma),\thet(\sigma),\thetah) }.
\end{equation}
\end{theorem*}

\bigskip

We formulate the above statement more precisely in the next section. Let us now comment the importance of the result. First, the function $L(\vv,\thet,\thetah)$ measures in a ``convex way" the distance of $(\vv,\thet)$ and $(\vc{0},\thetah)$, hence \eqref{MB-res} gives the exponential decay of such distance. Second, the relation \eqref{MB-res} shows that such distance is \emph{always decreasing}
exponentially to zero, which means that the steady solution is \emph{nonlinearly}\footnote{It means that all solutions are immediately attracted and not only the solutions being a~priori close to the steady one.} stable.

We would like to point out that the nonconstant heat capacity can be easily included and then the result holds as well, see Appendix~\ref{App:A}.
Further, a certain decay property can be observed also if $\delta$ appearing in \eqref{MBBounds} is equal to zero. Indeed, in such case we would need to replace $e^{-\lambda(\sigma-\tau)}$ by $(1+(\tau-\sigma))^{-\alpha}$ for some $\alpha>0$, see Appendix~\ref{App:B}.

The stability problem in the context of continuum thermodynamics belongs among very difficult problems, where not so much is known in case the fluid is heat conducting with material parameters depending on the temperature (which is the most typical case). Up-to-date the best result was obtained in \cite{DosPruRaj}, where the authors showed that for some nonnegative $H(\vv,\thet, \thetah)$ vanishing only for $\vv\equiv 0$ and $\thet=\thetah$,
$$
\intO{H(\vv(\tau),\thet(\tau), \thetah)} \to 0 \quad \textrm{ as } \quad \tau\to \infty,
$$
provided that the solution is sufficiently regular (at least strong solution) and the heat conductivity is constant. Hence, compared with this result, we obtained several significant improvements.

First, we do not require the regularity of the solution (note that in three dimensional setting the existence of a strong solution is not known), second, we are able to incorporate a general (nonconstant) heat conductivity and consequently also nonconstant heat capacity, last, we are able to provide the exponential decay and in fact the nonlinear stability result, which is not the case of \cite{DosPruRaj}. We would like to point out that our result similarly as the result in \cite{DosPruRaj} are inspired by the observation made already in \cite{BuMaPr19}, where much simplified problem was treated by using the energy-entropy relations. Note that such energy-entropy relations were already used in \cite{Dafermos} to study the stability problems. Furthermore, construction of specific energy like functionals inspired by energy-entropy relations was recently used for the so-called weak-strong uniqueness result for the fluid flow, see e.g. \cite{Fe1,Fe2}. However, all the above mentioned works deal with strong or classical solutions, or at least assume their existence. Note that in three dimensional setting, for nonlinear models such results are not available even for\footnote{Or equivalently, whenever the so-called energy equality holds.} $p\ge 11/5$. Last, but not least, the uniqueness results stated e.g.  in \cite{Fe1,Fe2} does not imply anything for the stability of the solution. It is important to emphasize that, contrary to above mentioned results, we are able to find and tune a~Lyapunov functional~$L$ that reflects the nonlinear nature of the problem and for which we just need to use weak solutions provided that $p\ge 11/5$. Note, that our result can be very straightforwardly extended also to two dimensional setting for $p\ge 2$.

There is also a second aspect of our result. There is a physical and also mathematical evidence that the incompressible Navier--Stokes--Fourier system should not be used as an adequate description of the fluid flow in case that the temperature gradient is large, e.g., in case that the boundary data are far from being constant. A canonical example is the problem how to describe the Rayleigh--B\'{e}nard convection, i.e., the instability in the fluid flow between to horizontal layers with prescribed large temperature gradients. The most typical model used in the physical and also mathematical community is the so-called Oberbeck--Boussinesq approximation, see e.g. \cite{RaRuSr96} for a class of potential models proper for description of the thermal convection. Furthermore, recent analytical results also suggests that the original Oberbeck--Boussinesq model must be accompanied by the proper boundary conditions, see \cite{BeFeOs22} or \cite{AbFe22} for the existence analysis. Note that the derivation of the Oberbeck--Boussinesq like approximations are done usually only formally and heuristically for classical solutions, or on a rigorous mathematical basis as in \cite{BeFeOs22} or \cite{AbFe22} but provided that the we do have a strong solution, which is not the case for long time intervals. However, noting like that was shown on the level of weak solutions, which is the only object whose existence is guaranteed. Therefore, it was not a~priori excluded, although not believed, that the instability may be a consequence of some potential singularity, which may appear due to the fact we deal only weak solutions. Hence, our result proves, that even on the level of weak solution the generalized Navier--Stokes--Fourier system is stable provided $p\ge 11/5$ in three dimensional setting and $p\ge 2$ in two dimensions, respectively.


In Section~\ref{S2} we formulate our result more precisely with all necessary technical details and in Section~\ref{S3} we provide the full rigorous proof.

\section{Rigorous statement of the result}\label{S2}
We denote $\mathbb{R}_+:=(0,+\infty)$ and $\mathbb{R}^{3\times 3}_{\textrm{sym}}$ be the set of symmetric $3\times 3$ matrices. We also use the standard notation for Lebesgue, Sobolev and Bochner spaces endowed with their classical norms. The symbols ``$\cdot$" and ``$:$" denote the scalar product in $\mathbb{R}^3$ and in $\mathbb{R}^{3\times 3}$. To shorten the notation, we also set
$$
W^{1,p}_{0,\diver}:=\overline{\left\{\vv \in \mathcal{C}^{\infty}_0(\Omega; \mathbb{R}^3); \, \diver \vv =0\right\}}^{\|\, \|_{1,p}} \qquad \textrm{ and } \qquad  L^2_{0,\diver}:=\overline{W^{1,2}_{0,\diver}}^{\|\, \|_2}
$$
for arbitrary $p\in [1,\infty)$. We also use the abbreviation $\fint$ for the mean value integral, i.e., $\fint_I := \frac{1}{|I|}\int_I$. Furthermore, the symbol $\S$ will be always understood as $\S(\thet, \D\vv)$ and $\eta$ always denotes $\ln \thet$ or vice versa, if $\eta$ is given then $\thet:= e^{\eta}$. Similarly, we set $\hat{\eta}:= \ln \thetah$. Also, in what follows we consider only the case $\delta:=1$ for simplicity. The case when $\delta>0$ can be done qualitatively similarly and one just has to trace the dependence of constants on the value of $\delta$. The case $\delta=0$ must be treated differently and we refer the interested reader to Appendix~\ref{App:B}.

\subsection{Class of proper solutions}
In this paper we consider only the class of \emph{proper} solutions to \eqref{MB1}--\eqref{MB5} and \eqref{MB8}.

\begin{definition}\label{proper-sol}
Let $\Omega \subset \mathbb{R}^3$ be the Lipschitz set. Assume that the conditions~\eqref{MBBounds} are satisfied and that $\thetah\in W^{1,2}(\Omega)$ fulfills $\tmin\le \thetah \le \tmax$ for some $0<\tmin\leq\tmax$ and solve \eqref{MB6} in a weak sense, i.e.,
\begin{equation}\label{thetahdf}
\intO{\kappa(\thetah)\nabla \thetah \cdot \nabla \varphi}=0 \qquad \textrm{ for all } \varphi \in W^{1,2}_0(\Omega).
\end{equation}
Let the initial data $\vv_0$ and $\thet_0$ be given and fulfill $\vv_0\in L^2_{0,\diver}$, $\thet_0\in L^1(\Omega)$ and  $\thet_0 \ge \tmin$ almost everywhere in $\Omega$. We say that $(\vv, \S, \thet, \eta)$ is a \emph{proper} solution if for all $T>0$:
\begin{itemize}[leftmargin=*]
\item we have
\begin{align}\label{MBap}
\begin{aligned}
\vv &\in \mathcal{C}([0,T]; L^2_{0,\diver}) \cap L^p(0,T; W^{1,p}_{0,\diver}),\quad \partial_t\vv \cap L^{p'}(0,T; (W^{1,p}_{0,\diver})^*),\\
\S &\in L^{p'}((0,T)\times \Omega),\qquad \thet \in \mathcal{C}([0,T]; L^1(\Omega)),\qquad \thet>0\qquad\mbox{a.e. in $(0,+\infty)\times\Omega$},\\
\eta & \in \mathcal{C}([0,T]; L^1(\Omega)) \cap L^2(0,T; W^{1,2}(\Omega)), \quad \eta-\hat{\eta} \in L^2(0,T; W^{1,2}_0(\Omega)),
\end{aligned}
\end{align}
\item the initial data are attained in the following sense
$$
\lim_{t\to 0_+} \|\vv(t)-\vv_0\|_2^2+ \|\thet(t)-\thet_0\|_1 = 0,
$$
\item
the equation~\eqref{MB1} is satisfied in the following sense
\begin{equation}\label{weak1}
\langle \partial_t\vv, \vw \rangle + \intO{(\S- \vv \otimes \vv): \nabla \vw}=0 \quad \textrm{ for all }\vw\in W^{1,p}_{0,\diver} \textrm{ and for a.a. } t\in (0,\infty);
\end{equation}
\item  the equation~\eqref{MB8} is satisfied in the following sense
\begin{equation}\label{entropy-limit}
\begin{split}
&-\intTO{\eta\dert\varphi}  - \intTO{\eta\,\vv\cdot \nabla\varphi}  + \intTO{\kappa(\thet)\nabla\eta\cdot\nabla\varphi} \\
&= \intTO{\frac{\S:\D{\vv}}{\thet}\,\varphi} + \intTO{\kappa(\thet)\,|\nabla \eta|^2\, \varphi} + \intO{\eta_0\, \varphi(0)}\\
&\qquad {}\qquad {}\qquad  \textrm{ for every $\varphi \in \mathcal{C}_0^\infty((-\infty, T)\times\Omega)$},
\end{split}
\end{equation}
\item
the following asymptotic property holds for any $T>0$
\begin{equation}\label{limit-k}
\lim_{k\to + \infty} \int_0^T\int_\Omega\chi_{\{\thet>k\}} k|\nabla \eta|^2\dx\dt=0.
\end{equation}
\end{itemize}
\end{definition}

The existence of a proper solution is in principle  established in \cite{ABK}.
In that article a solution is constructed such that it satisfies all properties of the proper solution except \eqref{limit-k}. Property~\eqref{limit-k} can be obtained by testing the heat equation with $\varphi=((\thet-k)/\thet)_+$. Unfortunately, we do not know if this is possible for any solution with regularity~\eqref{MBap} since the regularity of~$\nabla \thet$ is poor. Instead, one considers approximations of $\thet$ constructed in \cite{ABK}, tests the corresponding equation and derives the estimate. Convergence results from \cite{ABK} then allow to obtain \eqref{limit-k} also for the limiting function $\eta=\ln\thet$. Consequently, \cite{ABK} together with the mentioned procedure gives existence of the proper solution.

The property \eqref{limit-k} plays a role of a selector of the \emph{correct} solution. It will allow us to renormalize the equation for $\eta$ and consequently also for $\thet$ in the desired way, see Section~\ref{S3.4}. Note, that the renormalization procedure in this paper is slightly different than the usual methods used in the fluid mechanics since here, we also renormalize by a \emph{spatially} dependent function.

\subsection{Auxiliary functions}\label{sec:auxfnc}
Here, we introduce a notation for several auxiliary functions. For any $k\in \N$, we define the standard cut-off function
$$
\mathcal{T}_k(z):={\rm sign}(z) \min\{|z|,k\}.
$$
Next, we introduce a primitive function to the heat conductivity $\kappa$ and set
\begin{equation}\label{DFG}
\G(s):=\int_0^s \kappa(z)\, {\rm d} z.
\end{equation}
Since $\kappa$ is bounded from above and below (see the assumption \eqref{MBBounds}), the function $\G$ is strictly increasing and enjoys the following estimate
\begin{equation}\label{bound-G}
\underline{\kappa} s\leq \G(s) \leq \overline{\kappa} s  \qquad \textrm{ for all } s\ge 0.
\end{equation}
Then, with the help of above definitions, we define for any  $k\in \N$ and any $\alpha>0$
$$
\mathcal{F}_k(s):=\int_1^s \frac{\mathcal{T}_k(z)}{z} \, {\rm d}z \ \mbox{ and } \ \mathcal{F}_k^\al(s):=\int_1^s \frac{\mathcal{T}_k(z)}{z} (\G(\mathcal{T}_k(z)))^{-\al}\,{\rm d}z.
$$
It directly follows from the definition that we have the following convergence results
\begin{equation}\label{Fmk}
\begin{aligned}
\mathcal{F}_k(s) &\to s-1 &&\mbox{ for any } s>0, \mbox{ as } k\to +\infty,\\
\mathcal{F}_k^\al(s) &\to \mathcal{H}^\al(s):= \int_1^s (\G(z))^{-\al}\,{\rm d}z &&\mbox{ for any } s>0, \mbox{ as } k\to +\infty.
\end{aligned}
\end{equation}
Finally, for arbitrary $\beta\geq0$,  we define the function $L_{\beta}:\mathbb{R}^3\times(0,+\infty)\times(0,+\infty)\to\mathbb{R}$ by
\begin{equation}\label{LMB}
L_{\beta}(\vv,\thet,\thetah):= \beta|\vv|^2 + \thet-\thetah - (\mathcal{H}^\al(\thet) - \mathcal{H}^\al(\thetah))(\G(\thetah))^{\al}
\end{equation}
that measures a certain distance to the equilibria. 

Using the definition of $\mathcal{H}^{\al}$ we see that
$$
\frac{\partial L_{\beta}}{\partial{\thet}}(\vv,\thet,\thetah) = 1-\frac{(\G(\thetah))^{\al}}{(\G(\thet))^{\al}}, \qquad \frac{\partial^2 L_{\beta}}{\partial{\thet^2}}(\vv,\thet,\thetah) = \frac{\al \kappa (\theta)(\G(\thetah))^{\al}}{(\G(\thet))^{\al+1}}.
$$
Since $\G$ is strictly increasing and positive (for positive $\thet$) it is evident that for any fixed $\thetah>0$, $\al\geq0$ and $\beta\geq0$, the function $(\vv,\thet)\mapsto L_{\beta}(\vv,\thet,\thetah)$ is nonnegative and attains zero only for $(\vv,\thet)=(\vc{0},\thetah)$. Moreover, if $\beta>0$, the function $L_{\beta}(\vv,\thet,\thetah)$ is strictly convex in $(\vv,\thet)\in\mathbb{R}^3\times(0,+\infty)$.

\subsection{Result}
Having all necessary notation, we can now formulate the main theorem of the paper.
\begin{theorem}\label{TMB}
Let $\underline{\kappa}$, $\overline{\kappa}$, $\underline{\thet}$, $\overline{\thet}>0$ be given and fulfill $\tmin\leq\tmax$, $\underline{\kappa}\leq\overline{\kappa}$. Then there exists $\mu>0$ depending only on $\underline{\kappa}$ and $\Omega$ ($\mu$ is explicitly given below \eqref{muprv}) such that any proper solution satisfies  for all $\tau>\sigma\geq 0$
\begin{equation}\label{eq:twocrosses}
\|\vv(\tau)\|_2^2\le e^{-\mu (\tau-\sigma)}\|\vv(\sigma)\|_2^2.
\end{equation}
Moreover, let $\al\in (1/2,2/3]$ and  $R>0$ be arbitrary and fixed. Then, there exist $K,M>0$ such that any proper solution satisfies for all $\tau\geq\sigma\geq0$, $\lambda\in(0,\min(\mu,K))$ and $\beta:=2M\mu/(\mu-\lambda)$
\begin{equation}
  \|\vv_0\|_2^2 + \|\thet_0\|_1\le R \implies
  \intO{L_{\beta}(\vv(\tau),\thet(\tau),\thetah)} \le e^{-\lambda(\tau-\sigma)}\intO{L_{\beta}(\vv(\sigma),\thet(\sigma),\thetah)} . \label{assymptotictwo}
\end{equation}
The precise dependence of $K$ and $M$ on data of the problem is expressed in \eqref{rn28} and \eqref{rn29}.
\end{theorem}

Note that combining \eqref{eq:twocrosses} with \eqref{assymptotictwo} for~$\sigma=0$ gives existence of $C(\|\vv_0\|_2,\|\thet_0\|_1)$ such that for any~$\tau>0$
$$
\|\vv(\tau)\|_2^2\le C e^{-\mu \tau}\quad\mbox{and}\quad\intO{L_{\beta}(\vv(\tau),\thet(\tau),\thetah)} \le Ce^{-\lambda \tau}.
$$

\section{Proof of Theorem~\ref{TMB}}\label{S3}

This final section is devoted to the proof of the main theorem. In Section~\ref{S3.1}, we derive certain important algebraic inequalities. In Section~\ref{S3.2} we show the stability result for the velocity field. Then we focus on the temperature equation and introduce its time mollification in Section~\ref{S3.3}, which is used in Section~\ref{S3.3.5} to get the first a~priori estimates for the temperature. Finally, we derive rigorously the corner stone inequality \eqref{rn2} in Section~\ref{S3.4} and then in Section~\ref{S3.6} we finish the proof.

Since the definite value of $\delta>0$ plays no role from the point of the qualitative analysis and we are not interested in the dependence of constants on $\delta$, we, for the sake of readability, provide the proof only for $\delta=1$.

\subsection{Several algebraic inequalities} \label{S3.1}
In this preliminary part, we show some auxiliary inequalities. In what follows, we use the following notation
\begin{equation}\label{definice}
\begin{aligned}
f(\thet,\thetah)&:=\thet-\thetah - (\mathcal{H}^\al(\thet) - \mathcal{H}^\al(\thetah))(\G(\thetah))^{\al},\\
 g(\thet,\thetah)&:=\left| \left(\frac{\thet}{\thetah}\right)^\frac{1-\alpha}{2}-1\right|^2,\\
\bar{h}(\thet,\thetah)&:=\left(\frac{\G(\thet)}{\G(\thetah)}\right)^\frac{1-\alpha}{2}-1,
\end{aligned}
\end{equation}
where $\alpha\in (0,1)$ is given and $\thet,\thetah \in \mathbb{R}_+$ are arbitrary. Then, for the functions $f$ and $g$ we have the following:
\begin{lemma}\label{LM1}
Let $\kappa$ satisfy \eqref{MBBounds}. Then for any $\alpha\in (0,1)$ there exists a constant $C>0$ such that for all $\thet,\thetah>0$ we have
\begin{equation}\label{lemma}
0\le f(\thet,\thetah)\le C(\al,\underline{\kappa},\overline{\kappa})\thetah\left(\frac{\thet}{\thetah}\right)^{\al} \left(1+\left(\frac{\thetah}{\thet}\right)^{\al}+\left(\frac{\thetah}{\thet}\right)^{\frac{1-\al}{2}} \right)  g(\thet,\thetah).
\end{equation}
\end{lemma}
\begin{proof}
We can write
\begin{equation*}
\begin{split}
f(\thet,\thetah)&=\thet-\thetah - \G(\thetah)^{\al}\int_0^1\frac{{\rm d}}{{\rm d}s}\mathcal{H}^\alpha((\thet-\thetah)s+\thetah)\,{\rm d}s \\&
=(\thet-\thetah)\int_0^1\frac{(\G((\thet-\thetah)s+\thetah))^{\al}-(\G(\thetah))^{\al}}{(\G((\thet-\thetah)s+\thetah))^{\al}}\,{\rm d}s \\
&= (\thet-\thetah)\int_0^1\frac{1}{(\G((\thet-\thetah)s+\thetah))^{\al}} \left(\int_0^1\frac{{\rm d}}{{\rm d}t} (\G((\thet-\thetah)st+\thetah))^{\al}\,{\rm d}t \right)\,{\rm d}s\\
&=(\thet-\thetah)^2\int_0^1\frac{\al}{(\G((\thet-\thetah)s+\thetah))^{\al}} \left(\int_0^1 \frac{\kappa((\thet-\thetah)st+\thetah)}{(\G((\thet-\thetah)st+\thetah))^{1-\al}}  \, s\, {\rm d}t\right)\, {\rm d}s,
\end{split}
\end{equation*}
where we used the definition of $\G$ and $\mathcal{H}^{\alpha}$ (see \eqref{DFG} and \eqref{Fmk}). From the above identity, it directly follows that $f$ is nonnegative. Similarly, for $g$ we obtain
\begin{equation*}
g(\thet,\thetah)=\frac{\left|\thet^{\frac{1-\al}{2}}-\thetah^{\frac{1-\al}{2}} \right|^2}{|\thetah|^{1-\al}}=\frac{|\thet-\thetah|^2}{|\thetah|^{1-\al}}\left| \int_0^1 \frac{1-\al}{2 ((\thet-\thetah)s+\thetah)^{\frac{1+\al}{2}}} {\rm d}s\right|^2.
\end{equation*}
Therefore, using also the assumption \eqref{MBBounds}, we get for $\thet\neq \thetah$ that
\begin{equation*}
\begin{split}
\frac{f(\thet,\thetah)}{ g(\thet,\thetah)}&= \frac{4\al\thetah^{1-\al}}{(1-\al)^2\left|\int_0^1 ((\thet-\thetah)s+\thetah)^{\frac{-1-\al}{2}} {\rm d}s\right|^2}\\&\cdot \int_0^1 \int_0^1 \frac{ (\G((\thet-\thetah)st+\thetah))^{\al-1} \, \kappa((\thet-\thetah)st+\thetah) \, s}{\G((\thet-\thetah)s+\thetah)^{\al}}  {\rm d}t\,{\rm d}s\\
&\le \frac{4\overline{\kappa}\al\thetah^{1-\al}}{(1-\al)^2\left|\int_0^1 ((\thet-\thetah)s+\thetah)^{\frac{-1-\al}{2}} {\rm d}s\right|^2} \int_0^1 \int_0^1 \frac{ (\G((\thet-\thetah)st+\thetah))^{\al-1} \, s}{(\G((\thet-\thetah)s+\thetah))^{\al}}  {\rm d}t\,{\rm d}s\\
&= \frac{4\overline{\kappa}\al\thetah^{1-\al}}{(1-\al)^2\left|\int_0^1 ((\thet-\thetah)s+\thetah)^{\frac{-1-\al}{2}} {\rm d}s\right|^2} \int_0^1 \left(\int_0^s \frac{ (\G((\thet-\thetah)t+\thetah))^{\al-1}}{(\G((\thet-\thetah)s+\thetah))^{\al}}  {\rm d}t\right){\rm d}s\\
&\le \frac{4\overline{\kappa}\al\thetah^{1-\al}}{\underline{\kappa}(1-\al)^2\left|\int_0^1 ((\thet-\thetah)s+\thetah)^{\frac{-1-\al}{2}} {\rm d}s\right|^2} \int_0^1 \int_0^1 \frac{ ((\thet-\thetah)t+\thetah)^{\al-1}}{((\thet-\thetah)s+\thetah)^{\al}}  {\rm d}t\,{\rm d}s.
\end{split}
\end{equation*}
Since for any $\beta\in (0,1)$ we have
$$
\int_0^1 ((\thet-\thetah)s+\thetah)^{-\beta} {\rm d}s=\frac{(\thet)^{1-\beta}-(\thetah)^{1-\beta}}{(1-\beta)(\thet-\thetah)},
$$
we can substitute it to the above inequality to deduce the estimate
\begin{equation}\label{fg}
\begin{split}
\frac{f(\thet,\thetah)}{ g(\thet,\thetah)}&\le \frac{\overline{\kappa}\al\thetah^{1-\al}}{\underline{\kappa}\left|(\thet)^{\frac{1-\al}2}-(\thetah)^{\frac{1-\al}2} \right|^2}\frac{((\thet)^{1-\al}-(\thetah)^{1-\al})((\thet)^{\al}-(\thetah)^{\al})}{\al(1-\al)}\\
&= \frac{\overline{\kappa}\thetah(1+(\frac{\thetah}{\thet})^{\frac{1-\al}{2}})}
{\underline{\kappa}(1-\al)}\left(\frac{\thet}{\thetah}\right)^{\frac{1-\al}2}
\frac{|(\frac{\thet}{\thetah})^{\al}-1|}{|(\frac{\thet}{\thetah})^{\frac{1-\al}2}- 1| }
\end{split}
\end{equation}
Next, since $\alpha\in (0,1)$, we have the following estimate
$$
\frac{|(\frac{\thet}{\thetah})^{\al}-1|}{|(\frac{\thet}{\thetah})^{\frac{1-\al}2}- 1| } \le \left\{
\begin{aligned}
&C(\alpha)&&\textrm{ for }\frac{\thet}{\thetah}\in [0,2],\\
&C(\alpha)\left(\frac{\thet}{\thetah}\right)^{-\frac{1-3\al}2}  &&\textrm{ for }\frac{\thet}{\thetah}\ge 2.
\end{aligned}
\right.
$$
Hence, it follows from \eqref{fg} that
\begin{equation}\label{fg12}
\frac{f(\thet,\thetah)}{ g(\thet,\thetah)}\le \left\{
\begin{aligned}
&C(\alpha)\frac{\overline{\kappa}\thetah(1+(\frac{\thet}{\thetah})^{\frac{1-\al}{2}})}
{\underline{\kappa}(1-\al)}\leq C(\alpha)\frac{\overline{\kappa}}{\underline{\kappa}}\thetah,
&&\textrm{ for }\frac{\thet}{\thetah}\in [0,2],\\
&C(\alpha)\frac{\overline{\kappa}\thetah(1+(\frac{\thetah}{\thet})^{\frac{1-\al}{2}})}
{\underline{\kappa}(1-\al)}\left(\frac{\thet}{\thetah}\right)^{\al}
 &&\textrm{ for }\frac{\thet}{\thetah}\ge 2.
\end{aligned}
\right.
\end{equation}
The inequality \eqref{lemma} is then a direct consequence of \eqref{fg12}.

\end{proof}

\begin{lemma}\label{LM2}
Let $\kappa$ satisfy \eqref{MBBounds}. Then for any $\alpha\in (0,1)$ there exists a constant $C>0$ such that for all $\thet,\thetah>0$  we have
\begin{equation}\label{f-G}
f(\thet,\thetah)\leq  C(\al,\underline{\kappa},\overline{\kappa})\thetah\left(\frac{\thet}{\thetah}\right)^{\al} \left(1+\left(\frac{\thetah}{\thet}\right)^{\al}+\left(\frac{\thetah}{\thet}\right)^{\frac{1-\al}{2}} \right) \left| \bar{h}(\thet,\thetah)\right|^2.
\end{equation}
\end{lemma}

\begin{proof}

Using \eqref{bound-G} and \eqref{MBBounds}, it is not difficult to observe that
$$
C(\overline{\kappa}, \underline{\kappa}, \al) \frac{\partial}{\partial\thet} \left[ \left(\frac{\thet}{\thetah}\right)^{\frac{1-\al}{2}}-1\right]\leq  \frac{\partial\bar{h}(\thet,\thetah)}{\partial \thet}.
$$
Since $\bar{h}(\thet,\thetah)$ and $({\thet}/{\thetah})^{\frac{1-\al}{2}}-1$ are both zero when $\thet=\thetah$,  the integration leads to
$$
\sqrt{g(\thet,\thetah)}=\left |\left(\frac{\thet}{\thetah}\right)^{\frac{1-\al}{2}}-1 \right| \leq C(\overline{\kappa}, \underline{\kappa}, \al) |\bar{h}(\thet,\thetah)|.
$$
Consequently, raising to the power $2$ and employing \eqref{lemma}, we conclude the proof.
\end{proof}

\subsection{Uniform estimates for the velocity field} \label{S3.2}
Due to the sufficient regularity of $\vv$ we can set  $\vw:=\vv$ in \eqref{weak1} to deduce
\begin{equation}\label{st0}
 \frac{{\rm d}}{{\rm d}t} \frac{\|\vv\|^2_2}{2} = - \intO{\S:\D\vv} \mbox{ for all } t\in (0,T).
\end{equation}
Due to the nonnegativity of $\S:\D\vv$ this leads to the estimate
\begin{equation}\label{st1}
\intTOI{|\S:\D\vv|} \le \frac{\|\vv_0\|_2^2}{2}.
\end{equation}
Then, with the help of \eqref{MBBounds}, the embedding theorem and the Korn inequality, we deduce (recall we assume $\delta=1$ here)
\begin{equation}\label{premu}\begin{split}
\intO{|\vv|^2}&\leq C(\Omega)\intO{|\D\vv|^{2}}\le C(\Omega,\underline{\kappa})
\intO{\S:\D\vv},
\end{split}\end{equation}
and it follows from \eqref{st0} that
\begin{equation}\label{muprv}
\|\vv(t)\|^2_2 \leq \|\vv_0\|_2^2 \exp\{-\mu t\}  \mbox{ for all } t\in (0,T),
\end{equation}
where $\mu:=2/C(\Omega,\underline{\kappa})$ is a positive constant depending on $\Omega$ and $\underline{\kappa}$ via the constant $C(\Omega,\underline{\kappa})$ defined in \eqref{premu}.

\subsection{Reformulation of the entropy equality through the Steklov average in the time variable}\label{S3.3}
For any function $\varphi:(0,\infty)\times \Omega \to \mathbb{R}$ and any $t, h>0$, $x\in\Omega$ we denote
$$
 \varphi_h(t, x) = \fint_t^{t+h} \varphi(\tau, x) \,{\rm d}\tau.
$$
Similarly, for any $T,t,h>0$, $x\in\Omega$ and a function $\psi:(0,+\infty)\times\Omega\to\mathbb{R}$, such that $\supp \psi \subseteq (h, T-h)\times\Omega$, extended by zero outside of $(h,T-h)$, we introduce
$$
\psi_{-h}(t, x) = \fint_{t-h}^{t} \psi(\tau, x) \,{\rm d}\tau.
$$
Let us observe that
\begin{equation*}
\begin{split}
-\intTO{\eta\dert\psi_{-h}}  = - \frac{1}{h} \intTO{\eta(t)\, \psi(t)} + \frac{1}{h} \intTO{\eta(t)\, \psi(t-h)}.
\end{split}
\end{equation*}
Thus, using a change of variables in the second integral at the right-hand side and the fact that $\psi$ is supported on the interval $(h,T-h)$, we deduce
\begin{equation}
-\intTO{\eta\dert\psi_{-h}} = \intTO{\partial_t \fint_{t}^{t+h} \eta (\tau) {\rm d} \tau\, \psi(t)}= \intTO{\dert\eta_h\,\psi}.
\end{equation}
Consequently, setting $\varphi:=\psi_{-h}$ in \eqref{entropy-limit}, where $\psi \in \mathcal{C}_0^{\infty}((h,T-h)\times \Omega)$, and using the Fubini theorem it is easily seen that
\begin{equation*}
\begin{split}
&\intTO{\dert\eta_h\,\psi} + \intTO{(\nabla\eta\cdot \vv)_h \,\psi}  + \intTO{(\kappa(\thet)\nabla\eta)_h \cdot\nabla\psi}
\\
& = \intTO{\left(\frac{1}{\thet}\,\S:\D{\vv}\right)_h\,\psi} + \intTO{\left(\kappa(\thet)|\nabla\eta|^2\right)_h\,\psi}.
\end{split}
\end{equation*}
Finally, using the density and the weak$^*$ density argument and the a~priori information \eqref{MBap}, we see that the above identity holds also for $\psi:=u \varphi$, where $\varphi\in W^{1,\infty}(0,T)$ with $\supp\varphi\subset[h,T-h]$ and $u\in L^{\infty}(Q) \cap L^2(0,T; W_0^{1,2}(\Omega))$
\begin{equation}\label{entropy-limit2}
\begin{split}
&\intTO{\dert\eta_h\,u \varphi} + \intTO{(\nabla\eta\cdot \vv)_h \,u \varphi}  + \intTO{(\kappa(\thet)\nabla\eta)_h \cdot\nabla u \varphi}
\\
& = \intTO{\left(\frac{1}{\thet}\,\S:\D{\vv}\right)_h\,u \varphi} + \intTO{\left(\kappa(\thet)|\nabla\eta|^2\right)_h\,u \varphi}.
\end{split}
\end{equation}

\subsection{Uniform \texorpdfstring{$L^1$}{L} bound and minimum principle  for the temperature}\label{S3.3.5}
We consider $T, h>0$, $k\in\mathbb{N}$ and an arbitrary nonnegative $\varphi \in W^{1,\infty}(0,T)$ with $\supp\subset[h,T-h]$ and set
$$
u:= \mathcal{T}_k(\exp{\eta_h})\frac{(\exp{\eta_h}-\tmax)_+}{\exp{\eta_h}}
$$
in \eqref{entropy-limit2}. Here, we denoted $a_+:= \max\{a, 0\}$. Note that thanks to assumptions on the boundary data and \eqref{MBap}, we know that $u\in L^{\infty}(Q)\cap L^2(0,T; W^{1,2}_0(\Omega))$ and so $u$ is an admissible test function in \eqref{entropy-limit2}. Next, we evaluate the first term on the left hand side of \eqref{entropy-limit2}. Defining for $s\in\mathbb{R}$
$$
A_k(s):=\int_{-\infty}^s\mathcal{T}_k(\exp{\tau})\frac{(\exp{\tau}-\tmax)_+}{\exp{\tau}}\, {\rm d}\tau,
$$
we see that 
$0\leq A_k(s)\leq k(s-\ln \tmax)_+$ and
\begin{equation}\label{lim:ak}
A_k(s) \nearrow (\exp{s} - \tmax)_+ - \tmax (s-\ln \tmax)_+ \qquad \textrm{ as } k\to \infty.
\end{equation}
In addition, we can evaluate the first term in \eqref{entropy-limit2} by using integration by parts formula as
$$
\intTO{\dert\eta_h\,u\varphi}= \intTO{\dert A_k(\eta_h)\,\varphi}=-\intTO{ A_k(\eta_h)\,\dert\varphi}.
$$
Hence, we can let $h\to 0_+$ and by using the Lebesgue dominated convergence theorem, we observe that (recall that $\exp{\eta}=\thet$)
\begin{equation}\label{entropy-limit2.5}
\begin{split}
&-\intTO{ A_k(\eta)\,\dert\varphi} + \intTO{\nabla A_k(\eta)\cdot \vv  \varphi}  \\
&+\intTO{\kappa(\thet)|\nabla\eta|^2  \tmax \frac{T_k(\thet)}{\thet}\chi_{\{\thet >\tmax\}} \varphi}
\\
& = \intTO{\frac{\mathcal{T}_k(\thet)}{\thet}\,\S:\D{\vv}\frac{(\thet-\tmax)_+}{\thet} \varphi} + \intTO{\kappa(\thet)k|\nabla\eta|^2\frac{(\thet-\tmax)_+}{\thet} \chi_{\{\thet>k\}}\varphi}.
\end{split}
\end{equation}
Since $\diver \vv =0$, we have
$$
\intTO{\nabla A_k(\eta)\cdot \vv  \varphi} =0
$$
and the third term on the left hand side of~\eqref{entropy-limit2.5} is non negative.
Next, using~\eqref{limit-k}, we get for~$\|\varphi\|_\infty\leq 1$
$$
\intTO{\kappa(\thet)k|\nabla\eta|^2\frac{(\thet-\tmax)_+}{\thet} \chi_{\{\thet>k\}}\varphi}\le C(\overline{\kappa}, \tmax)\intTO{k|\nabla\eta|^2\chi_{\{\thet>k\}}}\to 0 \textrm{ as } k\to \infty.
$$
Similarly, using \eqref{st1}, we deduce for $\|\varphi\|_\infty\leq 1$
\begin{equation}\label{entropy-limit2.55}
\intTO{\frac{\mathcal{T}_k(\thet)}{\thet}\,\S:\D{\vv}\frac{(\thet-\tmax)_+}{\thet} \varphi}\le \intTO{\S:\D{\vv}}\le \frac{\|\vv_0\|_2^2}{2}.
\end{equation}
Thus, estimating the terms in \eqref{entropy-limit2.5} as suggested and passing to the limit as $k\to+\infty$, \eqref{entropy-limit2.5} and \eqref{lim:ak} imply that for $h>0$, $\tau\in(h,T-h)$ and $\varphi$ of the form
$$
\varphi(t)=\left\{
\begin{aligned}&\frac{t}{h} &&t\le h,\\
&1 &&t\in (h,\tau),\\
&\frac{\tau+h-t}{h} &&t\in (\tau, \tau+h),\\
&0 &&t\ge \tau+h
\end{aligned}
\right.
$$
we have
\begin{equation*}
\begin{split}
&\fint_{\tau}^{\tau+h}\intO{\left((\thet - \tmax)_+ - \tmax (\ln \thet-\ln \tmax)_+\right)}\dt \\
&\quad \le  \frac{\|\vv_0\|_2^2}{2}+\fint_0^h\intO{\left((\thet - \tmax)_+ - \tmax (\ln \thet-\ln \tmax)_+\right)}\dt.
\end{split}
\end{equation*}
Thanks to the continuity of $\thet$ in $L^1(\Omega)$, see \eqref{MBap}, we can let $h\to 0_+$ to get
\begin{equation}\label{entropy-limit2.6}
\begin{split}
&\frac12\intO{(\thet(\tau) - 2\tmax)_+}\le \intO{\left((\thet(\tau) - \tmax)_+ - \tmax (\ln \thet (\tau)-\ln \tmax)_+\right)} \le  \frac{\|\vv_0\|_2^2}{2}+\|\thet_0\|_1,
\end{split}
\end{equation}
where we used an algebraic inequality $(\thet(\tau) - 2\tmax)_{+} /2\le (\thet(\tau) - \tmax)_+ - \tmax (\ln \thet (\tau)-\ln \tmax)_+$. Finally, since $T>0$ was  arbitrary, it follows from \eqref{entropy-limit2.6} that for all $\tau>0$
\begin{equation}\label{ap-theta}
\begin{split}
&\|\thet(\tau)\|_1\le \|\vv_0\|_2^2+2\|\thet_0\|_1+ 2\tmax|\Omega|.
\end{split}
\end{equation}

To end this section, we can repeat the whole procedure with the function
$$
u:= \mathcal{T}_k(\exp{\eta_h})\frac{(\exp{\eta_h}-\tmin)_-}{\exp{\eta_h}}
$$
(here we denoted $a_-:= \min\{a, 0 \}$). By the very same arguments, with the exception that the term corresponding to \eqref{entropy-limit2.55} with $\varphi\geq0$ is non positive, and since we assume that $\thet_0\ge \tmin$, we get for almost all $(t,x)\in (0,+\infty)\times \Omega$ that
\begin{equation}\label{ap-theta2}
\begin{split}
\thet(t,x)\ge \tmin.
\end{split}
\end{equation}

%
%

\subsection{Renormalization of the entropy equality}\label{S3.4}

We continue with reading information from \eqref{entropy-limit2} but contrary to the preceding section, where we focused only on large values of $\thet$ (to avoid the problem with the boundary data), we introduce here a very specific test function $u$. For that purpose, we fix now arbitrary $\al\in(0,1)$, $T, h>0$ and $k\in \mathbb{N}$ sufficiently large and set
$$
u:=\mathcal{T}_k(\exp{\eta_h})\left[1- \left(\frac{\G(\thetah)}{\G(\mathcal{T}_k(\exp{\eta_h}))}\right)^{\alpha}\right]
$$
in \eqref{entropy-limit2}. Note that for $k>\tmax$, we have $u\in L^{\infty}(Q)\cap L^2(0,T; W^{1,2}_0(\Omega))$ and hence it is an admissible test function.
We also fix a nonnegative $\varphi\in W^{1,\infty}(0,T)$ with $\supp\varphi\in(0,T)$ and assume that $h>0$ is so small that $\supp\varphi\subset(h,T-h)$.

Our goal is to analyze the limit of \eqref{entropy-limit2} as $h\to 0_+$ and then as $k\to +\infty$ to deduce a renormalized version of the entropy equality that is suitable to prove the decay in time of the temperature.

We investigate each term in \eqref{entropy-limit2}. Using the definition of $\mathcal{F}_k$ and $\mathcal{F}_k^{\al}$, see Section~\ref{sec:auxfnc}, the time derivative term can be written as
\begin{equation*}\begin{split}
&\intTO{\dert\eta_h \, \mathcal{T}_k(\exp{\eta_h})\left[1- \left(\frac{\G(\thetah)}{\G(\mathcal{T}_k(\exp{\eta_h}))}\right)^{\al}\right] \varphi } \\&= \intTO{\dert\left( \mathcal{F}_k(\exp{\eta_h}) - \mathcal{F}_k^\al(\exp{\eta_h})(\G(\thetah))^{\al}\right)\varphi}\\
&= -\intTO{\left( \mathcal{F}_k(\exp{\eta_h}) - \mathcal{F}_k^\al(\exp{\eta_h})(\G(\thetah))^{\al}\right)\dert\varphi},
\end{split}\end{equation*}
where we used the integration by parts and the fact that $\varphi$ is compactly supported and also the fact that $\thetah$ is independent of time. Next, we want to let $h\to 0_+$ and also $k \to \infty$. First, thanks to the Jensen inequality, we have  $\exp \eta_h\leq \thet_h$. As $\thet\in L^{1}(Q)$ then $\thet_h \to \thet$ in $L^1(Q)$ and consequently also $\exp \eta_h \to \exp \eta =\thet$ in $L^1(Q)$. Consequently, as
\begin{equation}\label{est:forleb}
|\mathcal{F}_k(s)|+|\mathcal{F}_k^{\al}(s)| \le C(1+s)\quad\mbox{for $s>0$,}
\end{equation}
we can use the Lebesgue dominated convergence theorem to deduce that (recall also that $\thetah$ is bounded)
\begin{equation*}
\begin{split}
\lim_{h\to 0^+} &\left[-\intTO{\left( \mathcal{F}_k(\exp{\eta_h}) - \mathcal{F}_k^\al(\exp{\eta_h})(\G(\thetah))^{\al}\right)\dert\varphi}\right]\\
& = - \intTO{\left( \mathcal{F}_k(\thet) - \mathcal{F}_k^\al(\thet)(\G(\thetah))^{\al}\right)\dert\varphi}.
\end{split}
\end{equation*}
As \eqref{est:forleb} is actually for fixed $\alpha\in(0,1)$ uniform in $k\in\mathbb{N}$ we can use \eqref{Fmk} and  
the Lebesgue dominated convergence theorem to get
\begin{equation*}\begin{split}
&\lim_{k\to +\infty} \left[-\intTO{\left( \mathcal{F}_k(\thet) - \mathcal{F}_k^\al(\thet)(\G(\thetah))^{\al}\right)\dert\varphi}\right]\\
& = - \intTO{[\thet -1 - \mathcal{H}^\al(\thet)(\G(\thetah))^{\al}]\, \dert\varphi}.
\end{split}\end{equation*}
Since, $\thetah$ is independent of time the above relation can be rewritten as
$$
- \intTO{[\thet -\thetah - (\mathcal{H}^\al(\thet) - \mathcal{H}^\al(\thetah))(\G(\thetah))^{\al}]\, \dert\varphi}.
$$
We write the last integrand in this form, because it has nice properties described in Section~\ref{S3.1}.

Next, we focus on the term with $\S:\D \vv$. First, we consider the limit $h\to 0_+$ of
$$
\intTO{\left(\frac{1}{\thet}\,\S:\D{\vv}\right)_h\, \mathcal{T}_k(\exp{\eta_h})\left[1- \left(\frac{\G(\thetah)}{\G(\mathcal{T}_k(\exp{\eta_h}))}\right)^{\al}\right] \varphi}.
$$
Repeating the arguments for $\exp \eta_h$ already performed, also using that $\frac{1}{\thet}\,\S:\D{\vv}\in L^1((0,T)\times\Omega)$, we are in position to employ the dominated convergence theorem and take the limit under the integral sign, then
\begin{equation}\label{SD}\begin{split}
 &\intTO{\frac{1}{\thet}\,\S:\D{\vv}\, \mathcal{T}_k(\thet)\left[1- \left(\frac{\G(\thetah)}{\G(\mathcal{T}_k(\thet))}\right)^{\al}\right] \varphi}
\\& =\intTO{\chi_{\{\thet\leq k\}} \S:\D{\vv}\, \left[1- \left(\frac{\G(\thetah)}{\G(\thet)}\right)^{\al}\right] \varphi}\\&  + \intTO{\chi_{\{\thet> k\}}\frac{k}{\thet}\,\S:\D{\vv}\, \left[1- \left(\frac{\G(\thetah)}{\G(k)}\right)^{\al}\right] \varphi}.
\end{split} \end{equation}
Using that $\|\thetah\|_\infty\leq k$ for $k$ sufficiently large, it holds that
\begin{equation*}\begin{split}
&\left|\intTO{\chi_{\{\thet> k\}}\frac{k}{\thet}\,\S:\D{\vv}\, \left[1- \left(\frac{\G(\thetah)}{\G(k)}\right)^{\al}\right] \varphi}\right| \\
& \leq \intTO{\chi_{\{\thet> k\}}\S:\D{\vv}\, \left[1+ \left|\frac{\G(\thetah)}{\G(k)}\right|^{\al}\right] \varphi}
\leq
C(\overline{\kappa}, \underline{\kappa})\intTO{\chi_{\{\thet> k\}} \S:\D{\vv}\, \| \varphi\|_\infty}
\end{split} \end{equation*}
and  employing also \eqref{st1} it converges to zero when $k\to +\infty$. Thus taking the limit as $k\to +\infty$ in \eqref{SD} we get
$$
\intTO{  \S:\D{\vv}\, \left[1- \left(\frac{\G(\thetah)}{\G(\thet)}\right)^{\al}\right] \varphi}.
$$

Next, we focus on terms involving also $\nabla \thet$. For that purpose, we first observe that thanks to \eqref{limit-k} and the fact that $\nabla\eta\in L^2((0,T)\times\Omega)$, see \eqref{MBap}, we also have that for arbitrary $\varepsilon\in (0,1)$ (recall that $\eta=\ln \thet$)
$$
\int_0^{\infty} k^{-1-\varepsilon} \left(\intTO{\chi_{\{\thet>k\}}\frac{k|\nabla \thet|^2}{\thet^2}}\right) {\,{\rm d} {k}} <\infty.
$$
Consequently, using the Fubini theorem, we see that
$$
\begin{aligned}
\infty&>\int_0^{\infty} k^{-1-\varepsilon} \left(\intTO{\chi_{\{\thet>k\}}\frac{k|\nabla \thet|^2}{\thet^2} }\right) {\,{\rm d} {k}}\\
&=\intTO{\int_0^{\infty} \chi_{\{\thet>k\}}\frac{|\nabla \thet|^2}{k^{\varepsilon}\thet^2}  {\,{\rm d} {k}}}=\intTO{\int_0^{\thet} \frac{|\nabla \thet|^2}{k^{\varepsilon}\thet^2}  {\,{\rm d} {k}}}\\
&=\frac{1}{1-\varepsilon}\intTO{ \frac{|\nabla \thet|^2}{\thet^{1+\varepsilon}}  }
\end{aligned}
$$
Hence, using also \eqref{ap-theta2} we obtain for any $\varepsilon>0$
\begin{equation}\label{nabla}
\intTO{\frac{|\nabla \thet|^2}{\thet^{1+\varepsilon}}}\le C(T,\varepsilon). 
\end{equation}

Having such bound, we can focus on the terms with gradient of the temperature. We start with the one giving us the positive information, namely we focus on the elliptic term together with the second addend on the right-hand side of \eqref{entropy-limit2}, i.e., the term
\begin{equation*}
\begin{split}
&\intTO{(\kappa(\thet)\nabla\eta)_h \cdot\nabla \left[ \mathcal{T}_k(\exp{\eta_h})\left(1- \left(\frac{\G(\thetah)}{\G(\mathcal{T}_k(\exp{\eta_h}))}\right)^{\al}\right) \right]\varphi } \\
&- \intTO{\left(\kappa(\thet)\frac{|\nabla\thet|^2}{\thet^2}\right)_h\, \mathcal{T}_k(\exp{\eta_h})\left(1- \left(\frac{\G(\thetah)}{\G(\mathcal{T}_k(\exp{\eta_h}))}\right)^{\al}\right) \varphi }.
\end{split}
\end{equation*}
Since $\kappa(\thet)$ is bounded, $\nabla\eta\in L^2((0, T)\times\Omega)$, $\mathcal{T}_k(\exp{\eta_h})$ is uniformly bounded in $h>0$ and $\eta_h$ converges almost everywhere to $\eta$, we can take the limit as $h\to 0_+$ under the integral sign. After some manipulations we get
\begin{equation*}
\begin{split}
&\intTO{\kappa(\thet)\nabla \thet\cdot \left[\frac{\nabla (\mathcal{T}_k(\thet))}{\thet} -\frac{\mathcal{T}_k(\thet)\nabla\thet}{\thet^2}\right] \left[1- \left(\frac{\G(\thetah)}{\G(\mathcal{T}_k(\thet))}\right)^{\al}\right] \varphi} \\
&-\al\intTO{\nabla \G(\thet) \cdot \nabla\left[ \frac{\G(\thetah)}{\G(\mathcal{T}_k(\thet))}\right]\left( \frac{\G(\thetah)}{\G(\mathcal{T}_k(\thet))}\right)^{\al-1}\frac{\mathcal{T}_k(\thet)}{\thet}\,\varphi}=:A+B.
\end{split}
\end{equation*}
Employing that $\mathcal{T}_k(\thet)=\thet$ when $\thet\leq k$ and the boundedness of $\thetah$, then by virtue of assumption~\eqref{limit-k} it follows that
\begin{equation}
  |A|\leq  
  C \|\varphi\|_\infty \, \intTO{\chi_{\{\thet>k\}} \frac{|\nabla\thet|^2}{\thet^2}\, k } \to 0 \mbox{ as } k\to +\infty.
\end{equation}
The term $B$ can be split in the following way
\begin{equation}\label{B}
\begin{split}
B= -\al& \intTO{\chi_{\{\thet\leq k\}} \nabla \G(\thet) \cdot \nabla\left[ \frac{\G(\thetah)}{\G(\mathcal{T}_k(\thet))}\right]\left( \frac{\G(\thetah)}{\G(\mathcal{T}_k(\thet))}\right)^{\al-1}\,\varphi} \\
 &-\al\intTO{\chi_{\{\thet>k\}}{\nabla \G(\thet) \cdot \nabla\left[ \frac{\G(\thetah)}{\G(k)}\right]\left( \frac{\G(\thetah)}{\G(k)}\right)^{\al-1}\frac{k}{\thet}\,\varphi }}
\end{split}
\end{equation}
The second integral vanishes as $k\to +\infty$, indeed by a simple manipulation and the H\"{o}lder inequality we get that
\begin{equation*}
\begin{split}
&\left |\intTO{\chi_{\{\thet>k\}}{\frac{1}{\thet^\al}\nabla \G(\thet) \cdot \nabla \G(\thetah)\left( \frac{\G(k)}{\thet}\right)^{1-\al} \frac{1}{\G(\thetah)^{1-\al}}\,\frac{k}{\G(k)} \,\varphi }}\right| \\
&\leq C(\overline{\kappa}, \underline{\kappa}, \|\varphi\|_\infty)\left( \intTO{\chi_{\{\thet>k\}} \frac{|\nabla\thet|^2}{\thet^{2\al}}}\right)^{\frac{1}{2}} \left( \intTO{\chi_{\{\thet>k\}}\frac{ |\nabla\thetah|^2}{\G(\thetah)^{2(1-\al)}}}\right)^{\frac{1}{2}} \\
\end{split}
\end{equation*}
and, by virtue of the summability of $|\nabla\thet|^2/\thet^{2\al}$ for $\al\in(1/2, 1)$, see \eqref{nabla}, and of $|\nabla\thetah|^2$ we get the convergence to zero as $k\to +\infty$.
It remains to discuss the first integral term in \eqref{B}, which can be reformulated onto the whole domain $(0, T)\times\Omega$ because $\nabla \G(\mathcal{T}_k(\thet))=0$ when $\thet\geq k$. Through some manipulations and the integration by parts one gets
\begin{equation}\label{ell}
\begin{split}
-\al& \intTO{\chi_{\{\thet\leq k\}} \nabla \G(\thet) \cdot \nabla\left[ \frac{\G(\thetah)}{\G(\mathcal{T}_k(\thet))}\right]\left( \frac{\G(\thetah)}{\G(\mathcal{T}_k(\thet))}\right)^{\al-1}\,\varphi}\\
&=\al \intTO{ \nabla \G(\mathcal{T}_k(\thet)) \cdot \nabla\left[ \frac{\G(\mathcal{T}_k(\thet))}{\G(\thetah)}\right]\left( \frac{\G(\mathcal{T}_k(\thet))}{\G(\thetah)}\right)^{-1-\al}\,\varphi}\\
&=\al\intTO{ \G(\thetah) \,\nabla\left(\frac{\G(\mathcal{T}_k(\thet))}{\G(\thetah)}\right) \cdot \nabla\left( \frac{\G(\mathcal{T}_k(\thet))}{\G(\thetah)}\right)\left( \frac{\G(\mathcal{T}_k(\thet))}{\G(\thetah)}\right)^{-1-\al}\,\varphi}\\
&\quad + \al \intTO{  \nabla \G(\thetah) \cdot \nabla\left( \frac{\G(\mathcal{T}_k(\thet))}{\G(\thetah)}\right)\left( \frac{\G(\mathcal{T}_k(\thet))}{\G(\thetah)}\right)^{-\al}\,\varphi}\\
&=\frac{4\al}{(1-\al)^2}\intTO{ \G(\thetah) \left|\nabla\left(\frac{\G(\mathcal{T}_k(\thet))}{\G(\thetah)}\right)^{\frac{1-\al}{2}}\right|^2 \,\varphi}\\
&\quad + \frac{\al}{1-\al} \intTO{  \nabla \G(\thetah) \cdot \nabla\left(\left( \frac{\G(\mathcal{T}_k(\thet))}{\G(\thetah)}\right)^{1-\alpha}-1\right)\,\varphi}.
\end{split}
\end{equation}
Now, let us observe that the last integral vanishes because $\thetah$ is solution to \eqref{thetahdf} and the function $(( \G(\mathcal{T}_k(\thet))/\G(\thetah))^{1-\alpha}-1)$ belongs to $W^{1,2}_0(\Omega)$ for almost all $t\in (0,T)$ provided that $k\ge \tmax$.
%
Finally, in the first integral on the right hand side of~\eqref{ell} thanks to the almost everywhere  pointwise convergence  of $\mathcal{T}_k(\thet)$ to $\thet$ and by the nonnegativity of $\varphi$, and consequently of the whole integrand, we can employ the Fatou lemma and get that
\begin{equation*}
\begin{split}
&\frac{4\al}{(1-\al)^2}\intTO{ \G(\thetah) \left|\nabla\left(\frac{\G(\thet)}{\G(\thetah)}\right)^{\frac{1-\al}{2}}\right|^2 \,\varphi}\\
&\quad \leq \liminf_{k\to +\infty} \frac{4\al}{(1-\al)^2}\intTO{ \G(\thetah) \left|\nabla\left(\frac{\G(\mathcal{T}_k(\thet))}{\G(\thetah)}\right)^{\frac{1-\al}{2}}\right|^2 \,\varphi}.
\end{split}
\end{equation*}

Finally, we focus on the convective term
$$
\intTO{(\nabla\eta\cdot\vv)_h  \mathcal{T}_k(\exp{\eta_h})\left[1- \left(\frac{\G(\thetah)}{\G(\mathcal{T}_k(\exp{\eta_h}))}\right)^{\al}\right] \varphi }.
$$
Since $\nabla\eta\cdot\vv\in L^q((0, T)\times\Omega)$ for some $q>1$, it holds that $(\nabla\eta\cdot\vv)_h$ converges to $\nabla\eta\cdot\vv$ in $L^q((0, T)\times\Omega)$ with $q>1$; moreover $\eta_h$ converges to $\eta$ in any $L^q$ with $q\geq 1$ and $ \mathcal{T}_k(\exp{\eta_h})$ is uniformly bounded in $h>0$. Thus, the Lebesgue dominated convergence theorem ensures that
\begin{equation*}
\begin{split}
&\lim_{h\to 0^+} \intTO{(\nabla\eta\cdot\vv)_h\,  \mathcal{T}_k(\exp{\eta_h})\left[1- \left(\frac{\G(\thetah)}{\G(\mathcal{T}_k(\exp{\eta_h}))}\right)^{\al}\right] \varphi }\\
&=  \intTO{\nabla\eta\cdot\vv \, \mathcal{T}_k(\thet)\left[1- \left(\frac{\G(\thetah)}{\G(\mathcal{T}_k(\thet))}\right)^{\al}\right] \varphi }.
\end{split}
\end{equation*}
Next, we use the integration by parts and the fact that $\vv$ is divergence-free to observe
\begin{equation}\label{ct}
\begin{split}
 \intTO{&\nabla\eta\cdot\vv \, \mathcal{T}_k(\thet)\left[1- \left(\frac{\G(\thetah)}{\G(\mathcal{T}_k(\thet))}\right)^{\al}\right] \varphi }\\
 &=-\intTO{\nabla \thet \cdot\vv \, \frac{\mathcal{T}_k(\thet)}{\thet} \left(\frac{\G(\thetah)}{\G(\mathcal{T}_k(\thet))}\right)^{\al} \varphi }\\
&=-\intTO{\chi_{\{\thet>k\}}\vv \cdot \frac{k \nabla \thet}{\thet} \left(\frac{\G(\thetah)}{\G(\mathcal{T}_k(\thet))}\right)^{\al} \varphi} \\
&\quad +\intTO{\mathcal{T}_k(\thet)\,\vv  \cdot \nabla\left[\left(\frac{\G(\mathcal{T}_k(\thet))}{\G(\thetah)}\right)^{-\al}\right] \varphi }.
\end{split}
\end{equation}
For the first term, we use the assumption \eqref{limit-k}, the H\"{o}lder inequality and the estimate \eqref{bound-G} to deduce
\begin{equation*}
\begin{split}
&\left|\intTO{\chi_{\{\thet>k\}}\vv \cdot \frac{k \nabla \thet}{\thet} \left(\frac{\G(\thetah)}{\G(\mathcal{T}_k(\thet))}\right)^{\al} \varphi }\right| \\
&\qquad \le C(\overline{\kappa}, \underline{\kappa},\tmax,\tmin,\varphi)\|\vv\|_2 k^{\frac12 - \alpha}\left(\intTO{\chi_{\{\thet>k\}}\frac{k |\nabla \thet|^2}{\thet^2}}
\right)^{\frac{1}{2}} \to 0 \quad \textrm{ as } k\to\infty,
\end{split}
\end{equation*}
provided, we chose $\al\geq\frac12$. For the second term on the right hand side, we use the Young inequality, \eqref{bound-G} and properties of $\thetah$ to conclude
\begin{equation}\label{ct2}
\begin{split}
&\left|\intTO{\mathcal{T}_k(\thet)\,\vv  \cdot \nabla\left[\left(\frac{\G(\mathcal{T}_k(\thet))}{\G(\thetah)}\right)^{-\al}\right] \varphi }\right|\\
&\quad \le \frac{\alpha}{2}\intTO{\mathcal{G}(\thetah)\left(\frac{\G(\mathcal{T}_k(\thet))}{\G(\thetah)}\right)^{-1-\al}\left|\nabla\frac{\G(\mathcal{T}_k(\thet))}{\G(\thetah)}
\right|^2 \varphi }\\
&\qquad +\frac{\alpha}{2}\intTO{\frac{(\mathcal{T}_k(\thet))^2}{(\G(\mathcal{T}_k(\thet)))^{1+\al}}|\vv|^2 (\G(\thetah))^{\al}\varphi }.
\end{split}
\end{equation}
Please observe here that the first term on the right hand side will be absorbed by the first term on the right hand side of \eqref{ell}. For the second term, we use \eqref{bound-G}, the uniform bound \eqref{ap-theta}  as well as the Korn inequality to get
\begin{equation}
\begin{aligned}
&\frac{\alpha}{2}\intTO{\frac{(\mathcal{T}_k(\thet))^2}{(\G(\mathcal{T}_k(\thet)))^{1+\al}}|\vv|^2 (\G(\thetah))^{\al}\varphi}\\
&\le \frac{\alpha \overline{\kappa}^{\al}\tmax^{\al}}{2\underline{\kappa}^{1+\al}}\intTO{\thet^{1-\al}|\vv|^2 \varphi }\le \frac{\alpha \overline{\kappa}^{\al}\tmax^{\al}}{2\underline{\kappa}^{1+\al}}\int_0^T \|\thet\|_1^{1-\al} \varphi \left(\intO{|\vv|^{\frac{2}{\al}} }\right)^{\al}\, \dt\\
&\le \frac{C(\Omega)\alpha \overline{\kappa}^{\al}\tmax^{\al}}{2\underline{\kappa}^{1+\al}}(\|\vv\|_0^2+2\|\thet_0\|_1+4\overline{\kappa}|\Omega|)^{1-\al}\int_0^T \|\D \vv\|_2^{2} \varphi\, \dt,
\end{aligned}\label{nevim}
\end{equation}
provided that $\al\in [1/3,1)$.


Collecting all terms, using \eqref{MBBounds} for the last term in \eqref{nevim} and the fact that any $\varphi\in W^{1,\infty}_0(0,T)$ is a pointwise limit of $\varphi_n\in W^{1,\infty}(0,T)$ such that $\supp\varphi_n\in(0,T)$ and $\|\varphi_n\|_{1,\infty}\leq 2\|\varphi\|_{1,\infty}$, we obtain
\begin{equation}\label{rn}
\begin{split}
& - \intTO{[\thet -\thetah - (\mathcal{H}^\al(\thet) - \mathcal{H}^\al(\thetah))\G(\thetah)^{\al}]\, \dert\varphi}\\
 &+  \frac{\al}{2} \intTO{\left | \nabla \frac{\G(\thet)}{\G(\thetah)} \right|^2  \G(\thetah) \left( \frac{\G(\thet)}{\G(\thetah)}\right)^{-1-\al}\,\varphi}\\& \leq \intTO{  \S:\D{\vv}\, \left[1- \left(\frac{\G(\thetah)}{\G(\thet)}\right)^{\al}\right] \varphi} \\
 &\qquad +\frac{C(\Omega)\alpha \overline{\kappa}^{\al}\tmax^{\al}}{2\underline{\kappa}^{2+\alpha}}(\|\vv_0\|_2^2+2\|\thet_0\|_1+4\overline{\kappa}|\Omega|)^{1-\al}\intTO{\S:\D \vv\,\varphi}
\end{split}
\end{equation}
for every nonnegative $\varphi\in W^{1,\infty}_0(0, T)$ and for all $\al \in (\frac12, 1)$. Recall that the bound for $\al$ is given by the limit passage in the convective term, see \eqref{ct}, and in the elliptic term, see below \eqref{B}.

We finish this part by the proper choice of the function $\varphi$ that finally gives the desired decay estimates. For $\tau,\sigma\in[0,T]$, $\tau>\sigma$ and a parameter $l\geq 2/(\tau-\sigma)$ let us introduce the function defined as
\begin{equation}
  \label{phiL}\varphi_l(t)=\begin{cases}
0 &\mbox{ if } 0\leq t\leq\sigma,\\
l(t-\sigma) &\mbox{ if } \sigma\leq t\leq\sigma+\frac{1}{l},\\
1 &\mbox{ if } \sigma+\frac{1}{l}\leq t\leq \tau-\frac{1}{l},\\
l(\tau - t) &\mbox{ if }\tau-\frac{1}{l}\leq t\leq \tau\\
0 &\mbox{ if }\tau\leq t\leq T.
\end{cases}
\end{equation}
We also fix a suitable $\lambda>0$. The proper choice of $\lambda$ is done later. Then, we set $\varphi(t):=\exp{\{\lambda t\}}\, \varphi_l(t)$ in \eqref{rn}. The time derivative term is equal to
\begin{equation}\label{time}
\begin{split}
 &- \frac{1}{l}\int_\sigma^{\sigma+\frac{1}{l}}\intO{[\thet -\thetah - (\mathcal{H}^\al(\thet) - \mathcal{H}^\al(\thetah))(\G(\thetah))^{\al}]\exp{\{\lambda t\}}}\dt\\
 &+  \frac{1}{l}\int_{\tau-\frac{1}{l}}^\tau\intO{[\thet -\thetah - (\mathcal{H}^\al(\thet) - \mathcal{H}^\al(\thetah))(G(\thetah))^{\al}]\exp{\{\lambda t\}}}\dt\\
 & - \lambda\intTO{[\thet -\thetah - (\mathcal{H}^\al(\thet) - \mathcal{H}^\al(\thetah))(\G(\thetah))^{\al}]\, \exp{\{\lambda t\}}\,\varphi_l}.
\end{split}
\end{equation}
Since, $\thet\in C((0, T); L^1(\Omega))$, we have  $\mathcal{H}^\al(\thet)\in C((0, T); L^1(\Omega))$ as well, and consequently the first two addends in \eqref{time} converge to
$$
\begin{aligned}
&\intO{\left[\thet(\tau) -\thetah - (\mathcal{H}^\al(\thet(\tau)) - \mathcal{H}^\al(\thetah))(\G(\thetah))^{\al}\right]\exp{\{\lambda \tau\}}} \\
&\qquad - \intO{[\thet(\sigma) -\thetah - (\mathcal{H}^\al(\thet(\sigma)) - \mathcal{H}^\al(\thetah))(\G(\thetah))^{\al}]\exp\{\lambda\sigma\}}
\end{aligned}
$$
when $l\to +\infty$. The convergence in the third term in \eqref{time} as well as the convergence in all remaining terms in \eqref{rn} directly follows from the facts that $\varphi_l\leq 1$ and converges point-wisely to $\chi_{(\sigma,\tau)}$ for $l\to +\infty$, $\S:\D\vv\in L^1((0, T)\times\Omega)$ and $1- ({\G(\thetah)}/{\G(\thet)})^{\al}$ is bounded, and the Lebesgue dominated convergence theorem. Since $T>0$ was arbitrary, the final inequality has for any $\tau>\sigma\geq0$ the form
\begin{equation}\label{rn2}
\begin{split}
& \intO{\left[\thet(\tau) -\thetah - (\mathcal{H}^\al(\thet(\tau)) - \mathcal{H}^\al(\thetah))(\G(\thetah))^{\al}\right]\exp{\{\lambda \tau\}}}\\ &-\lambda \int_\sigma^{\tau} \intO{[\thet -\thetah - (\mathcal{H}^\al(\thet) - \mathcal{H}^\al(\thetah))(\G(\thetah))^{\al}]\, \exp{\{\lambda t\}}}\, \dt\\
&+  \frac{\al}{2} \int_\sigma^{\tau}\intO{\left| \nabla \frac{\G(\thet)}{\G(\thetah)} \right|^2  \G(\thetah) \left(\frac{\G(\thet)}{\G(\thetah)}\right)^{-1-\al}\exp{\{\lambda t\}}}\, \dt \\
&
\leq \int_\sigma^{\tau}\intO{  \S:\D{\vv}\, \left[1- \left(\frac{\G(\thetah)}{\G(\thet)}\right)^{\al}\right] \exp{\{\lambda t\}}}\, \dt \\
&\qquad +\frac{C(\Omega)\alpha \overline{\kappa}^{\al}\tmax^{\al}}{2\underline{\kappa}^{2+\alpha}}(\|\vv_0\|_2^2+2\|\thet_0\|_1+4\overline{\kappa}|\Omega|)^{1-\al}\int_\sigma^{\tau}  \intO{\S:\D \vv}\exp{\{\lambda t\}}\, \dt\\
&\qquad +\intO{[\thet(\sigma) -\thetah - (\mathcal{H}^\al(\thet(\sigma)) - \mathcal{H}^\al(\thetah))(\G(\thetah))^{\al}]\exp{\{\lambda\sigma\}}},
\end{split}
\end{equation}
with $\alpha\in(1/2,1)$ and $\lambda >0$. It is the starting point for the further investigation.

\subsection{Exponential decay and stability - proof of the main result}\label{S3.6}

The statement~\eqref{eq:twocrosses} was proved in Section~\ref{S3.2}. To prove~\eqref{assymptotictwo} we start with the estimate for~$f$ and $\overline{h}$ (see \eqref{definice}), which is based on the Sobolev embedding. Notice that $\overline{h}(\thet,\thetah)=0$ on $\partial\Omega$. Using Lemma~\ref{LM2}, the fact that $\thetah$ is bounded,  the minimum principle \eqref{ap-theta2}, the H\"{o}lder and the Poincar\'{e} inequalities and the a~priori estimate \eqref{ap-theta}, we deduce that for $\alpha \in [1/3,2/3]$, there holds
\begin{equation}\label{Holder}
\begin{split}
&\intO{f(\thet,\thetah)}\le C(\al,\underline{\kappa},\overline{\kappa})\overline{\thet}^{1-\al}\left(\frac{
\overline{\thet}}{\underline{\thet}}\right)^{\al}\intO{\thet^{\al}|\overline{h}(\thet,\thetah)|^2}\\
&\leq C(\al,\underline{\kappa},\overline{\kappa})\overline{\thet}^{1-\al}\left(\frac{
\overline{\thet}}{\underline{\thet}}\right)^{\al} \|\thet\|_{1}^\al \|\overline{h}(\thet,\thetah)\|_{\frac{2}{1-\al}}^2\\
&\leq C(\Omega,\al,\underline{\kappa},\overline{\kappa})\overline{\thet}^{1-\al}\left(\frac{
\overline{\thet}}{\underline{\thet}}\right)^{\al} (2\|\thet_0\|_{1}+\|\vv_0\|_2^2 + 2\overline{\thet}|\Omega|)^\al \|\nabla \overline{h}(\thet,\thetah)\|_{2}^2\\
&\leq K^{-1} \frac{\al}{2} \intO{\left| \nabla \frac{\G(\thet)}{\G(\thetah)} \right|^2  \G(\thetah) \left(\frac{\G(\thet)}{\G(\thetah)}\right)^{-1-\al}},
\end{split}
\end{equation}
where we defined
\begin{equation}\label{rn28}
K:= \left[C(\Omega,\al,\underline{\kappa},\overline{\kappa})\overline{\thet}^{1-\al}\left(\frac{
\overline{\thet}}{\underline{\thet}}\right)^{\al} (2\|\thet_0\|_{1}+\|\vv_0\|_2^2 + 2\overline{\thet}|\Omega|)^\al\frac2{\alpha\G(\underline\thet)}
\right]^{-1}.
\end{equation}
Hence, defining
\begin{equation}\label{rn29}
M:=\frac{C(\Omega)\alpha \overline{\kappa}^{\al}\tmax^{\al}}{2\underline{\kappa}^{2+\al}}(\|\vv_0\|_2^2+2\|\thet_0\|_1+4\overline{\kappa}|\Omega|)^{1-\al}+1,
\end{equation}
and using \eqref{Holder} in \eqref{rn2}, we get for $\alpha\in (1/2,2/3]$
\begin{equation}\label{rn3}
\begin{split}
& \intO{f(\thet(\tau),\thetah)\exp{\{\lambda \tau\}}}+(K-\lambda) \int_\sigma^{\tau} \intO{f(\thet(\tau),\thetah)\, \exp{\{\lambda t\}}}\, \dt\\
&\qquad \leq M\int_\sigma^{\tau}\intO{  \S:\D{\vv} \exp{\{\lambda t\}}}\, \dt  +\intO{f(\thet(\sigma),\thetah)\exp\{\lambda\sigma\}}.
\end{split}
\end{equation}

Note that $K$ and $M$ depend on the size of the data. We use the identity \eqref{st0}, integration by parts and \eqref{premu} to get an estimate
$$
\begin{aligned}
&\int_\sigma^{\tau}\intO{  \S:\D{\vv} \exp{\{\lambda t\}}}\, \dt= - \frac12  \int_\sigma^{\tau} \frac{{\rm d}}{{\rm d}t}  \|\vv\|_2^2  \exp{\{\lambda t\}}\, \dt\\
&=- \frac12\|\vv(\tau)\|_2^2  \exp{\{\lambda \tau\}}+   \frac12\|\vv(\sigma)\|_2^2\exp{\{\lambda \sigma\}} + \frac\lambda2 \int_\sigma^{\tau}  \|\vv\|_2^2  \exp{\{\lambda t\}}\, \dt\\
&\le- \frac12\|\vv(\tau)\|_2^2  \exp{\{\lambda \tau\}}+   \frac12\|\vv(\sigma)\|_2^2\exp{\{\lambda \sigma\}} + \lambda\frac {C(\Omega,\underline\kappa)}2 \int_\sigma^{\tau} \intO{\S:\D{\vv}  \exp{\{\lambda t\}}}\, \dt,
\end{aligned}
$$
that gives with help of definition of $\mu$ below \eqref{muprv}
$$
\left(1-\frac\lambda\mu\right)\int_\sigma^{\tau}\intO{  \S:\D{\vv} \exp{\{\lambda t\}}}\, \dt
\le- \frac12\|\vv(\tau)\|_2^2  \exp{\{\lambda \tau\}}+   \frac12\|\vv(\sigma)\|_2^2\exp{\{\lambda \sigma\}}.
$$

Substituting this estimate with $\lambda\in(0,\mu)$ into \eqref{rn3}, we get
\begin{equation}\label{rn35}
\begin{split}
  & \frac{M\mu}{2(\mu-\lambda)}\|\vv(\tau)\|_2^2  \exp{\{\lambda \tau\}}+\intO{f(\thet(\tau),\thetah)\exp{\{\lambda \tau\}}}\\
  &\qquad+(K-\lambda) \int_\sigma^{\tau} \intO{f(\thet(\tau),\thetah)\, \exp{\{\lambda t\}}}\, \dt\\
&\qquad \leq \frac{M\mu}{2(\mu-\lambda)}\|\vv(\sigma)\|_2^2  \exp{\{\lambda \sigma\}}+\intO{f(\thet(\sigma),\thetah)\exp\{\lambda\sigma\}},
\end{split}
\end{equation}
which is for $\lambda\in(0,\min(K,\mu))$ the statement \eqref{assymptotictwo}.

\cite{BuMaPr19} \cite{Dafermos} \cite{Fe1,Fe2}

\begin{appendix}

\section{Possible generalizations}\label{App:A}
Here, we give just several short notes, how the result can be extended onto more general models. We do not provide complete proofs but we rather present ideas how one can proceed.

\subsection{Nonconstant heat capacity}
In general, the equation \eqref{MB3} has the form
\begin{equation}
\dert e(\thet) + \diver (\vv e(\thet)) - \diver (\kappa(\thet)\nabla \thet) = \S(\thet,\D\vv) : \D \vv, \label{MB333}
\end{equation}
where $e:\mathbb{R}_+ \to \mathbb{R}_+$ is increasing function. In addition if we assume that
$$
\underline{\kappa}\le e'(\thet)\le \overline{\kappa},
$$
then the function $e$ has Lipschitz inverse and the equation \eqref{MB333} can be rewritten into the form for unknown function $\Theta=e(\thet)$ (re-scaled temperature), $\Theta:(0,+\infty)\times\Omega\to(0,+\infty)$ as
\begin{equation}
\dert \Theta + \diver (\vv \Theta) - \diver (\tilde{\kappa}(\Theta)\nabla \Theta) = \tilde\S(\Theta,\D\vv) : \D \vv, \label{MB334}
\end{equation}
where $\tilde\kappa(\Theta)=\kappa(e_{-1}(\Theta))e_{-1}'(\Theta)$ and $\tilde\S(\Theta,\D\vv)=\S(e_{-1}(\Theta),\D\vv)$. Similarly, we neet to replace $\S(\thet,\D\vv)$ by $\tilde\S(\Theta,\D\vv)$ in \eqref{MB1}.
Hence, after the rescaling of the temperaturewe are in the same position as in \eqref{MB1}--\eqref{MB3}, the only change is a different heat conductivity $\tilde{\kappa}$ and an extra stress tensor $\tilde \S$. Note that $\tilde{\kappa}$ and $\tilde \S$ satisfy the assumptions of Theorem~\ref{TMB}.

\subsection{Notes to the case \texorpdfstring{$\delta=0$}{d}}\label{App:B}
First, if $\delta=0$ the procedure in Section~\ref{S3.2} must be modified since $\|\vv\|_2^r\leq C\intO{ \S:\D{\vv}}$ and consequently \eqref{st0} turns into
\begin{equation}\label{st0A}
  \frac{{\rm d}}{{\rm d}t} \frac{\|\vv\|^2_2}{2} + \mu \|\vv\|^r_2 \le 0\mbox{ for all } t\in (0,T)
\end{equation}
with $\mu=1/C(\Omega,\underline\kappa)$.
Assuming that $\|\vv(\tau)\|_2>0$ for all $\tau\in (0,t)$ we can deduce from \eqref{st0A} (recall that $r\ge 2$)
\begin{equation}\label{st0B}
 \frac{{\rm d}}{{\rm d}t} \|\vv\|^{2-r}_2\ge  \mu(r-2)  \mbox{ for all } \tau\in (0,t),
\end{equation}
which implies
\begin{equation}\label{st0C}
\|\vv(t)\|_2\le   \frac{\|\vv_0\|_2}{(1+\mu(r-2)t\|\vv_0\|^{r-2}_2)^{\frac{1}{r-2}}}.
\end{equation}

Next difference arises in estimating the last term on the right hand side of the second line in \eqref{nevim} where one needs to estimate $\|\vv\|_{2/\alpha}^2$ by means of $\|\D{\vv}\|_r^r\leq \intO{\S:\D{\vv}}/\underline\kappa$. Consequently, the resulting inequality \eqref{rn3} is changed to (here we require that $\al\in[2(3-r)/3r,1]$ and $\nu=(3r\alpha-6+2r)/(3r-6+2r)$ is a number arising from the interpolation of $L^{2/\al}$ between $L^2$ and $W^{1,r}$ norm)
\begin{equation}\label{rn3A}
\begin{split}
& \intO{f(\thet(\tau),\thetah)\exp{\{\lambda \tau\}}}+(K-\lambda) \int_0^{\tau} \intO{f(\thet(\tau),\thetah)\, \exp{\{\lambda t\}}}\, \dt\\
&\qquad \leq M\int_0^{\tau}\left(\intO{  \S:\D{\vv}}\right)^{\frac{2-\nu}{r}}\|\vv\|_2^{\nu}\exp{\{\lambda t\}} \, \dt  +\intO{f(\thet_0,\thetah)}.
\end{split}
\end{equation}
Using the H\"{o}lder inequality, the identity \eqref{st1} and also \eqref{st0C}, we have
\begin{equation}\label{rn32A}
\begin{aligned}
\int_0^{\tau}&\left(\intO{  \S:\D{\vv}}\right)^{\frac{2-\nu}{r}}\|\vv\|_2^{\nu}\exp{\{\lambda t\}} \, \dt \\
&\le C\int_0^{\tau}\left(\intO{  \S:\D{\vv}}\right)^{\frac{2-\nu}{r}}\frac{\exp{\{\lambda t\}}}{(1+t)^{\frac{\nu}{r-2}}} \, \dt\\
&\le C\left(\int_0^{\tau}\intO{  \S:\D{\vv}}\, \dt \right)^{\frac{2-\nu}{r}}\left(\int_0^{\tau}\frac{\exp{\{\lambda t\frac{r}{r-2+\nu}\}}}{(1+t)^{\frac{\nu}{r-2}\frac{r}{r-2+\nu}}} \, \dt\right)^{\frac{r-2+\nu}{r}}\\
&\le \tilde{C}(\vv_0)\frac{\exp{\{\lambda \tau\}}}{(1+\tau)^{\frac{\nu}{r-2}}}.
\end{aligned}
\end{equation}
In the last estimate in \eqref{rn32A}, we assumed that $\lambda$ is so large that the function
$${\exp{\{\lambda t\frac{r}{r-2+\nu}\}}}/{(1+t)^{\frac{\nu}{r-2}\frac{r}{r-2+\nu}}}$$
is increasing on $(0,+\infty)$, and
$$
\int_0^{\tau}\frac{\exp{\{\lambda t\frac{r}{r-2+\nu}\}}}{(1+t)^{\frac{\nu}{r-2}\frac{r}{r-2+\nu}}} \, \dt
\leq 2\int_{\frac\tau2}^{\tau}\frac{\exp{\{\lambda t\frac{r}{r-2+\nu}\}}}{(1+t)^{\frac{\nu}{r-2}\frac{r}{r-2+\nu}}} \, \dt\leq \frac{C}{(1+\tau)^{\frac{\nu}{r-2}\frac{r}{r-2+\nu}}}\int_{\frac\tau2}^{\tau}\exp{\{\frac{\lambda t r}{r-2+\nu}\}} \, \dt.
$$
By using \eqref{rn32A} in  \eqref{rn3A} we get the following decay
\begin{equation}\label{rn33A}
\begin{split}
& \intO{f(\thet(\tau),\thetah)}\le \frac{C(\vv_0,\thet_0)}{(1+\tau)^{\frac{\nu}{r-2}}}.
\end{split}
\end{equation}

\end{appendix}

\providecommand{\bysame}{\leavevmode\hbox to3em{\hrulefill}\thinspace}
\providecommand{\MR}{\relax\ifhmode\unskip\space\fi MR }
\providecommand{\MRhref}[2]{%
  \href{http://www.ams.org/mathscinet-getitem?mr=#1}{#2}
}
\providecommand{\href}[2]{#2}


\end{document}